\documentclass[11pt]{amsart}

\usepackage{amssymb,amsmath,accents}
\usepackage{bbm}
\usepackage{a4wide}

\usepackage[colorlinks=true, pdfstartview=FitV, linkcolor=blue, 
citecolor=blue]{hyperref}

\DeclareMathAlphabet{\mymathbb}{U}{bbold}{m}{n}

\newtheorem{theorem}{Theorem}[section]

\newtheorem{lemma}[theorem]{Lemma}
\newtheorem{fact}[theorem]{Fact}
\newtheorem{coro}[theorem]{Corollary}
\theoremstyle{definition}

\newtheorem{example}[theorem]{Example}
\newtheorem{remark}[theorem]{Remark}

\newcommand{\ts}{\hspace{0.5pt}}
\newcommand{\nts}{\hspace{-0.5pt}}

\newcommand{\RR}{\mathbb{R}\ts}
\newcommand{\CC}{\mathbb{C}\ts}
\newcommand{\NN}{\mathbb{N}}
\newcommand{\ZZ}{\mathbb{Z}}

\newcommand{\cM}{\mathcal{M}}

\newcommand{\ee}{\ts\mathrm{e}}
\newcommand{\ii}{\ts\mathrm{i}}

\newcommand{\one}{\mymathbb{1}}
\newcommand{\nix}{\mymathbb{0}}

\newcommand{\vt}{\vartheta}

\newcommand{\trans}{{\raisebox{1pt}{$\scriptscriptstyle \mathsf{T}$}}}

\newcommand{\exend}{\hfill$\Diamond$}

\newcommand{\diag}{\mathrm{diag}}
\newcommand{\comm}{\mathrm{comm}}
\newcommand{\lcm}{\mathrm{lcm}}
\newcommand{\tr}{\mathrm{tr}}
\newcommand{\Mat}{\mathrm{Mat}}
\newcommand{\bs}[1]{\boldsymbol{#1}}

\newcommand{\defeq}{\mathrel{\mathop:}=}

\newcommand{\myfrac}[2]{\frac{\raisebox{-2pt}{$#1$}}
  {\raisebox{0.5pt}{$#2$}}}

\begin{document}

\title{Embedding of sub-stochastic matrices}

\author{Michael Baake}
\address{Fakult\"at f\"ur Mathematik, Universit\"at Bielefeld, 
         Postfach 100131, 33501 Bielefeld, Germany}
\email{mbaake@math.uni-bielefeld.de}
         
\author{Kiah Swinsburg}         
\address{School of Natural Sciences, Discipline of Mathematics,
          University of Tasmania,
     \newline \indent PO Box 807, Sandy Bay, TAS 7006, Australia}
 \email{Kiah.Swinsburg@utas.edu.au}

\begin{abstract} 
  The classic embedding problem for finite-dimensional Markov matrices
  has a natural counterpart for sub-stochastic matrices, which is
  analysed and discussed here. One necessary and sufficient
  characterisation of embeddability can be given via the unique
  extension of a sub-stochastic matrix to a stochastic one with one
  added state in conjunction with the embedding of this extension.
  This is then explicitly treated for $d\leqslant 3$.
\end{abstract}


\maketitle

\vspace*{5mm}

\section{Introduction}

A real matrix with non-negative entries is \emph{stochastic} or
\emph{Markov} when all row sums are $1$, and \emph{sub-stochastic} (or
\emph{sub-Markov}) when they are $\leqslant 1$.  Likewise, a real
matrix with non-negative off-diagonal entries is called a
\emph{stochastic} or \emph{Markov generator}, which is also known as a
\emph{rate matrix}, when all row sums are $0$. If they are
$\leqslant 0$, one speaks of a \emph{sub-stochastic} (or
\emph{sub-Markov}) \emph{generator}.  We shall use the terms
synonymously.

For the compatibility of a given Markov matrix $M$ with a
continuous-time Markov chain, one needs to know whether $M$ is
\emph{embeddable}, which means that one has $M=\ee^Q$ for some Markov
generator $Q$; see \cite{Elfving, King} for origins and \cite{Davies,
  BS3} with references therein for background. More precisely, this is
the embedding problem for the time-homogeneous case. It can be
extended to situations with time-dependent generators, but we will not
consider this here.

The corresponding embedding question also emerges for sub-stochastic
matrices. In the time-homogeneous case, they are only compatible with
continuous time when they possess a real logarithm that is a
sub-stochastic generator. Here, we analyse this problem, which is
closely related to the classic embedding problem because every
genuinely sub-stochastic matrix (generator) can be turned into a
Markov matrix (generator) by adding one state; see
\cite[Ch.~XV{\!}.8]{Feller}.  This added state is called the
\textsf{c}-state (for cemetery or coffin) because it is absorbing; see
\cite[Ch.~III]{KS} for background.  \smallskip

This short paper is organised as follows. In Section~\ref{sec:prelim},
we recall and adapt some results from linear algebra and matrix
analysis that we need for the embedding problem, some of which we also
prove for the convenience of the reader. Then,
Theorem~\ref{thm:extend} relates the sub-stochastic embedding to its
stochastic counterpart with the added $\mathsf{c}$-state, which is
followed by examples and the explicit treatment of the embedding for
$d\leqslant 3$.

\section{Preliminaries}\label{sec:prelim}

Let us begin by recalling the following connection.

\begin{fact}\label{fact:Q-to-M}
  If\/ $Q$ is a sub-stochastic generator, $M=\ee^Q$ is a
  sub-stochastic matrix. When\/ $Q$ is Markov, then so is\/ $M$.
\end{fact}

\begin{proof}
  Let $Q$ be sub-stochastic and recall that
  $M = \lim_{n\to\infty} \bigl( \one + \frac{1}{n} Q \bigr)^n$ by
  Euler's theorem. Here, $\one + \frac{1}{n} Q$ has non-negative
  entries for all sufficiently large $n$, hence also its $n\ts$th
  power.  As non-negativity is preserved in the limit, $M$ has
  non-negative entries only.
  
  Now, $Q$ being sub-stochastic entails $Q\bs{1} \leqslant \bs{0}$, to
  be understood element-wise for the column vectors, which also
  implies
  $\bigl( \one + \frac{1}{n} Q \bigr) \bs{1} \leqslant \bs{1}$. When
  the matrix in brackets is non-negative, the inequality can be
  iterated, and we get $M \bs{1} \leqslant \bs{1}$ in the limit as
  $n\to\infty$.
  
  When $Q$ is Markov, we have $Q \bs{1} = \bs{0}$, which implies
  $M\bs{1} = \bs{1}$, and $M$ is Markov.
\end{proof}

Let $\cM_d$ denote the set of $d$-dimensional Markov matrices and
$\cM_{d,\leqslant}$ the super-set of sub-stochastic matrices. Both
sets are convex and (topologically) closed.  Now, there is a unique
way to extend any $M \in \cM_{d,\leqslant}$ to a Markov matrix
$M_{\mathsf{c}} \in \cM_{d+1}$ by adding one state, compare
\cite[Thm.~II.3.3]{Chung}, and analogously for generators $Q$, which
leads to the block matrices
\begin{equation}\label{eq:extend}
     M_{\mathsf{c}} \, = \, \left( \begin{array}{c|c}
     M & u \\ \hline 0 \cdots 0 & 1 \end{array} \right)
         \quad \text{and} \quad
     Q_{\mathsf{c}} \, = \, \left( \begin{array}{c|c}
     Q & v  \\ \hline 0 \cdots 0 & 0 \end{array} \right) .
\end{equation}
Here, $u$ (respectively $v$) stands for a $d$-dimensional column
vector whose entries are uniquely determined by the demand that all
row sums of $M_{\mathsf{c}}$ be $1$ (respectively those of
$Q_{\mathsf{c}}$ be $0$), so
\begin{equation}\label{eq:vectors}
    u \, = \, (\one - M) \bs{1} \quad \text{and} \quad
    v \, = \, - Q \ts \bs{1} \ts .
\end{equation}
Let us refer to the particular block structure of the extension in
\eqref{eq:extend} as a \emph{$\mathsf{c}$-block structure}. The
relation of a matrix to its $\mathsf{c}$-block extension has a number
of interesting consequences, as we shall now explain.

To do so, we need to refer to the (complex) \emph{Jordan normal form}
(JNF) of a matrix and to its elementary Jordan blocks, denoted as
$J_m (\lambda)$. The latter is called \emph{non-trivial} when its
dimension satisfies $m \geqslant 2$. The JNF will become central for
analysing the possible real matrix logarithms, as we shall see after
Fact~\ref{fact:Culver} below. Prior to this, let us discuss some
spectral relations, where eigenvalues in the spectrum occur with their
algebraic multiplicities, meaning that we treat spectra as multi-sets.

\begin{fact}\label{fact:spectrum}
  Let\/ $M \in \cM_{d, \leqslant}$ be a sub-stochastic matrix, and\/
  $M_{\mathsf{c}} \in \cM_{d+1}$ its unique extension with the\/
  $\mathsf{c}$-block structure from $\eqref{eq:extend}$.  Then, their
  spectra, viewed as multi-sets, are related by\/
  $\sigma (M_{\mathsf{c}}) = \sigma (M) \cup \{ 1 \}$.  Likewise, for
  a sub-stochastic generator\/ $Q$ and its extension\/
  $Q_{\mathsf{c}}$, one has\/
  $\sigma (Q_{\mathsf{c}} ) = \sigma (Q) \cup \{ 0 \}$, again as
  multi-sets.
\end{fact}

\begin{proof}
  Simply observe that the $\mathsf{c}$-block structure
  $M_{\mathsf{c}}$, via a Laplace expansion along its last row,
  implies the identity
\begin{equation}\label{eq:char-poly}
     \det ( M_{\mathsf{c}} - \lambda \ts \one_{d+1} ) \, = \,
     (1-\lambda) \ts \det (M - \lambda \ts \one_{d} ) \ts ,
\end{equation}
and analogously for $Q_{\mathsf{c}}$ versus $Q$, then with a
pre-factor $-\lambda$ on the right-hand side. These relations between
the characteristic polynomials prove the spectral claims.
\end{proof}

Let us be more specific on the characteristic and minimal
polynomials as follows.

\begin{lemma}\label{lem:min-poly}
  Consider\/ $M\in\cM_{d,\leqslant}$ and let\/ $p$ and\/ $q$ be the
  characteristic and the minimal polynomial of\/ $M$,
  respectively. If\/ $p^{}_{\mathsf{c}}$ and\/ $q^{}_{\mathsf{c}}$ are
  the corresponding polynomials of the unique extension of\/ $M$ to\/
  $M_{\mathsf{c}}\in\cM_{d+1}$, one has\/
  $p^{}_{\mathsf{c}} (\lambda) = (1-\lambda) \ts p(\lambda)$ together
  with
\[
  q^{}_{\mathsf{c}} (\lambda) \, = \,
       \lcm ( q(\lambda), 1-\lambda)
       \, = \, \begin{cases}
       q(\lambda) \ts , &  \text{if\/ $1 \in \sigma (M)$}  , \\
       (1-\lambda) \ts q(\lambda) \ts , & \text{otherwise} \ts .
       \end{cases}
\]
Likewise, the corresponding polynomials for\/ $Q$ and\/
$Q_{\mathsf{c}}$ from \eqref{eq:extend} satisfy\/
$p^{}_{\mathsf{c}} (\lambda) = -\lambda \, p(\lambda)$ and\/
$q^{}_{\mathsf{c}} (\lambda) = \lcm ( q(\lambda), -\lambda )$.
\end{lemma}

\begin{proof}
  The claim on the characteristic polynomials is clear from
  Eq.~\eqref{eq:char-poly} for $M_{\mathsf{c}}$, and from the comment
  in the proof of Fact~\ref{fact:spectrum} for the generator case.
  
  To determine the minimal polynomial of $M_{\mathsf{c}}$, its block
  structure implies that $q^{}_{\mathsf{c}}$ must annihilate $M$ and
  possess $1$ as a root, so both $1-\lambda$ and $q(\lambda)$ must be
  factors of $q^{}_{\mathsf{c}} (\lambda)$, as is the case. The least
  common multiple has minimal degree among all polynomials with these
  factors, so it remains to check that it really annihilates
  $M_{\mathsf{c}}$.
  
  To this end, let $q^{}_{\mathsf{c}} (\lambda) = \sum_{m=0}^{N}
  \alpha_m \lambda^m$ and observe that we get
\[
    q^{}_{\mathsf{c}} (M_{\mathsf{c}}) \, = \, \begin{pmatrix}
       q^{}_{\mathsf{c}} (M) & \widetilde{u} \\
       0 \cdots 0 & q^{}_{\mathsf{c}} (1) \end{pmatrix} \, = \,
     \begin{pmatrix} \nix & \widetilde{u} \\ 0 \cdots 0 & 0
     \end{pmatrix}
\]  
with $\widetilde{u} = \sum_{m=1}^{N} \alpha_m
(\one + M + \ldots + M^{m-1}) u$, where $u = (\one - M) \bs{1}$
as in \eqref{eq:vectors}. This gives
\[
    \widetilde{u} \, = \sum_{m=1}^{N} \alpha_m (\one - M^m) \bs{1}
    \, = \, \bigl( q^{}_{\mathsf{c}} (1) - \alpha^{}_{0} \bigr) \bs{1}
    + \bigl( \alpha^{}_{0} \one - q^{}_{\mathsf{c}} (M) \bigr) \bs{1}
    \, = \, \bs {0}
\]
due to the properties of $q^{}_{\mathsf{c}}$ already established. We
thus have $q^{}_{\mathsf{c}} (M_{\mathsf{c}}) = \nix$ as claimed.

For $Q_{\mathsf{c}}$, we have $q^{}_{\mathsf{c}} (\lambda) = \sum_{m=1}^{N}
  \alpha_m \lambda^m$, because $\alpha^{}_{0} = 0$, and we obtain
\[
    q^{}_{\mathsf{c}} (Q_{\mathsf{c}}) \, = \, \begin{pmatrix}
       q^{}_{\mathsf{c}} (Q) & \widetilde{v} \\
       0 \cdots 0 & q^{}_{\mathsf{c}} (1) \end{pmatrix} \, = \,
     \begin{pmatrix} \nix & \widetilde{v} \\ 0 \cdots 0 & 0
     \end{pmatrix}
\]
with $\widetilde{v} = \sum_{m=1}^{N} \alpha_m Q^{m-1} v = -
\sum_{m=1}^{N} \alpha_m Q^m \bs{1} = q^{}_{\mathsf{c}} (Q) \bs{1}
= \bs{0}$ by \eqref{eq:vectors}, giving the desired conclusion
for the generator case.
\end{proof}

\begin{remark}
  The formula for the minimal polynomial is the one expected for
  matrices in block-diagonal form, which we do not have. Indeed, as
  any non-trivial Jordan block shows, it cannot hold in general when
  off-diagonal blocks are non-zero. In our case, the vectors $u$ and
  $v$ have a special form, as used in the proof. More generally, one
  obtains
  $q^{}_{\mathsf{c}} (\lambda) = \lcm ( q(\lambda), r-\lambda)$ for
  all $\mathsf{c}$-block matrices $
  R_{\mathsf{c}} \, = \, \left( \begin{smallmatrix}  R & w \\
                                  0 \cdots 0 & r \end{smallmatrix}
  \right) $ with $w = (r - R) \ts \bs{1}$, which means that all row
  sums of $R_{\mathsf{c}}$ are $r$.  Indeed, one then has
\[
    R_{\mathsf{c}}^m \, = \, \begin{pmatrix}
    R^m & w_m \\ 0 \cdots 0 & r^m  \end{pmatrix}
\]
with $w^{}_{0} = \bs{0}$ and
$w_m = \sum_{\ell=0}^{m-1} r^{\ell} R^{m-1-\ell} w$ for $m\in\NN$.
With the special form of $w$, one can then verify that
$q^{}_{\mathsf{c}}$ annihilates $R_{\mathsf{c}}$; further details are
left to the interested reader.  \exend
\end{remark}

Let us formulate one particular consequence.

\begin{coro}\label{coro:JNF}
  If\/ $\lambda = 1$ is an eigenvalue of a sub-stochastic matrix\/
  $M$, the geometric multiplicity of\/ $\lambda$ equals its algebraic
  one, so no non-trivial Jordan blocks for\/ $\lambda$ can occur.
\end{coro}

\begin{proof}
  When $M$ is a Markov matrix, $1$ is always an eigenvalue, and the
  claim is a well-known result; see \cite[Fact~2.2]{BS3} and
  references given there.

  So, let $M \in \cM_{d,\leqslant}$ be sub-stochastic, and let
  $1\in\sigma (M)$, with algebraic degree $k \geqslant 1$. Then,
  $1 \in\sigma (M_{\mathsf{c}})$, with algebraic degree $k+1$ as a
  consequence of Lemma~\ref{lem:min-poly}. As the corresponding factor
  of the minimal polynomial for both matrices is $(\lambda - 1)$, the
  geometric degree of $1$ for $M$ is also $k$, which proves the claim.
\end{proof}

Corollary~\ref{coro:JNF} has the following interesting counterpart,
where we use the notion of irreducibility in the sense of non-negative
matrices; compare \cite[Sec.~15.1]{Lan}.

\begin{fact}\label{fact:radius}
  Let\/ $M$ be a sub-stochastic matrix that is also irreducible.
  Then, $M$ is Markov if and only if its spectral radius is\/ $1$.
\end{fact}

\begin{proof}
  When $M$ is Markov, $\bs{1}$ is a right eigenvector with eigenvalue
  $1$, which is maximal and thus the spectral radius, by standard
  arguments for Markov matrices; see \cite[Ch.~15.8]{Lan}.

  To continue, let $\sigma_i = (M \bs{1})_i$ be the $i\ts$th row sum
  of $M$ and assume that $M$ is sub-stochastic, with at least one row
  sum being strictly smaller than $1$, so $\min_{i} \sigma_{i}<1$ and
  $\max_{i}\sigma_{i}\leqslant 1 $.  Let $\rho$ be the spectral radius
  of $M$, which is an eigenvalue of $M$ by \cite[Thm.~15.5.1]{Lan}.  A
  standard result from the theory of non-negative matrices, see
  \cite[Sec.~15.3 and Exc.~15.3.7]{Lan}, states that we have a
  dichotomy, namely either
  $\min_{i} \sigma_i < \rho < \max_{i} \sigma_i$ or
  $\min_{i} \sigma_{i} = \max_{i} \sigma_{i} = \rho$.  In either case,
  we get $\rho<1$, and our claim follows.
\end{proof}

A matrix $B\in \Mat (d,\RR)$ is called \emph{cyclic} (or
\emph{non-derogatory}, see \cite{HJ}) when its characteristic
polynomial is also its minimal one. This property is equivalent to all
eigenspaces of $B$ being one-dimensional, as well as to its
\emph{commutant}
$\comm (B) = \{ C \in \Mat (d,\RR) : [B, C] = \nix \}$ being
Abelian. In fact, one then has $\comm (B) = \RR [B]$, which is the
polynomial ring generated by $B$; see \cite[Thm.~2.1]{BS3} for a
formulation of the underlying theorem by Frobenius in this
setting. Non-commutativity of $\comm(B)$ is relevant for potential
non-uniqueness of the real logarithm of $B$. In particular, one has to
consider further matrix logarithms emerging from similarity transforms
with invertible matrices that commute with $B$ but not with
$\log (B)$.

\smallskip

To see how things develop in various dimensions, we will need further
observations from linear algebra, and one result from matrix analysis
on the existence of real matrix logarithms.  The latter was analysed
by Culver \cite{Culver}. We summarise his results as follows.

\begin{fact}[Culver]\label{fact:Culver}
  A matrix\/ $B\in\Mat (d,\RR)$ has a real logarithm if and only if
  the following two conditions are satisfied.
\begin{enumerate}\itemsep=2pt
\item The matrix\/ $B$ is non-singular.
\item Each elementary Jordan block of the JNF of\/ $B$ that belongs 
  to a negative eigenvalue occurs with even multiplicity.
\end{enumerate}   
Further, when all eigenvalues of\/ $B$ are positive real numbers and
no elementary Jordan block occurs twice, the real logarithm of\/ $B$ 
is unique.  \qed
\end{fact}

For the origin and possible types of non-uniqueness, we refer to the
original paper \cite{Culver}.  We are now ready to treat the embedding
question.

\section{Embedding results}\label{sec:embed}

The embedding for $d=1$ is trivial, as any positive number is the
exponential of its (existing) logarithm. We thus restrict to
$d\geqslant 2$, and begin with the following general connection.

\begin{theorem}\label{thm:extend}
  A sub-stochastic matrix\/ $M\in\cM_{d,\leqslant}$ can only be
  embeddable when it satisfies\/ $\det (M) > 0$. Then, one has\/
  $M = \ee^Q$, with a sub-stochastic generator\/ $Q$, if and only if
  the unique extension of\/ $M$ to\/ $M_{\mathsf{c}}$ is embeddable
  with a stochastic generator in $\mathsf{c}$-block form, where the
  latter then is the unique extension\/ of\/ $Q$ to\/
  $Q_{\mathsf{c}}$.
\end{theorem} 

\begin{proof}
  The first claim is a consequence of $\det (\ee^Q) =
  \ee^{\tr (Q)} > 0$ for any real matrix $Q$.

  If $M = \ee^Q$, there is the unique extension $Q_{\mathsf{c}}$ of
  $Q$ given in \eqref{eq:extend}, which is a Markov generator in
  dimension $d\ts {+}1$.  Its exponential, $\exp (Q_{\mathsf{c}})$, is
  an element of $\cM_{d\ts +1}$ by Fact~\ref{fact:Q-to-M}, and it has
  $\mathsf{c}$-block form with $M$ as its upper-left block. Since the
  extension of $M$ to a Markov matrix $M_{\mathsf{c}}$ in
  $\mathsf{c}$-block form is unique, we get
  $M_{\mathsf{c}} = \exp (Q_{\mathsf{c}})$.
   
  Conversely, if $M_{\mathsf{c}}$ is embeddable with a generator in
  $\mathsf{c}$-block form, $Q_{\mathsf{c}}$ say, the restriction to
  the upper-left block gives $M = \ee^Q$ by standard arguments on the
  block structure, where $Q$ is the upper-left block of
  $Q_{\mathsf{c}}$ and thus a sub-stochastic generator by
  construction.
\end{proof}

It may be instructive to check the last step of the proof explicitly,
with $v$ from \eqref{eq:vectors}. Then, $\exp (Q_{\mathsf{c}})$ has
the proper $\mathsf{c}$-block structure, with $M=\ee^Q$ as stated and
the column vector
\[
  u \, = \sum_{n\geqslant 1} \myfrac{1}{n \ts !}  Q^{n-1} v
  \, = \, - \sum_{n\geqslant 1} \myfrac{1}{n\ts !} Q^n \bs{1}
  \, = \, ( \one - M ) \bs{1} \, \geqslant \, \bs{0} \ts ,
\]
where the last step follows from $M \bs{1} \leqslant \bs{1}$.  So,
every embeddable sub-stochastic matrix is a block in some embeddable
Markov matrix, for which there are further possibilities as follows.

\begin{remark}
  More generally, one can consider a Markov block matrix of the form
\[
  M \, = \, \begin{pmatrix} M_1 & B \\ 0 & M_2 \end{pmatrix} ,
\]
where $M_2$ is itself Markov, while $M_1$ is sub-stochastic.  This
relates $M_1$ to the general setting of absorbing Markov chains;
compare \cite[Ch.~III]{KS}. If $M$ is embeddable with a Markov
generator $Q$ of matching block form, it is clear that $M_2$ is
embeddable as well, and $M_1$ is embeddable via the corresponding
block of $Q$, which then is a sub-stochastic generator.

Clearly, one can also ask for extensions of doubly sub-stochastic
matrices, where some interesting additional features emerge. We do not
discuss this here, but refer to \cite[Ch.~2]{MOA} for some background
and references.  \exend
\end{remark}

\subsection{Results for \texorpdfstring{$d=2$}{\textit{d}=3}}

When $d=2$, the most general Markov matrix is
$M = \left( \begin{smallmatrix} 1-a & a \\ b & 1-b \end{smallmatrix}
\right)$ with $a,b \in [0,1]$. Here, Kendall's theorem, compare
\cite{King,BS3}, states that $M$ is embeddable if and only if
$\det (M) >0$. In this case, the embedding is unique, and given by the
principal matrix logarithm, which reads
\[
     Q \, = \, - \frac{\log(1-a-b)}{a+b} (M-\one) \ts ,
\]     
with the limiting case $Q=\nix$ for $a=b=0$.

Since $\det (\ee^R) = \ee^{\tr (R)}$, we are in the particularly
simple situation that the necessary condition $\det (M) > 0$ is also
sufficient. Let us thus see what happens for a sub-stochastic matrix,
which we parameterise as
\[
    M \, = \, \begin{pmatrix} a & c \\ d & b \end{pmatrix},
\]
where $a,b,c,d \in [0,1]$ together with $a+c \leqslant 1$ and
$b+d \leqslant 1$. As $M$ is non-negative, its spectral radius is an
eigenvalue \cite[Thm.~15.5.1]{Lan}, so $\det (M) > 0$ implies that
both eigenvalues are positive,
\[
  \lambda^{}_{\pm} \, = \, \myfrac{1}{2} \bigl(
  t \pm \sqrt{t^2 - 4 \Delta} \bigr) \, > \, 0 \ts ,
\]
where $t = \tr (M) = a+b$, $\Delta = \det (M) = ab-cd$ and
$t^2 - 4 \Delta = (a-b)^2 + 4 cd \geqslant 0$.

Let us look at a specific real matrix logarithm of a non-negative
matrix, namely the principal one. Some standard computations based on
\cite{Higham} lead to the following.

\begin{lemma}\label{lem:sub-Kendall}
  Let\/ $M \in \Mat (2,\RR)$ be a matrix with non-negative entries
  and\/ $\det (M)>0$. Then, the spectrum is positive, and the
  principal matrix logarithm of\/ $M$ is the real matrix
\[
  \log (M) \, = \,
  \begin{cases}
    \frac{\lambda^{}_{+} \log(\lambda^{}_{-}) -
    \lambda^{}_{-} \log(\lambda^{}_{+})}
    {\lambda^{}_{+} - \ts \lambda^{}_{-}} \ts \one +
    \frac{\log(\lambda^{}_{+}) - \log(\lambda^{}_{-})}
    {\lambda^{}_{+} - \ts \lambda^{}_{-}} M \ts , &
  \text{if $\lambda^{}_{+} > \lambda^{}_{-}$} > 0, \\
   (\log(\lambda) - 1) \ts \one + \frac{1}{\lambda} M \ts ,
  & \text{if $\lambda^{}_{+} = \lambda^{}_{-} = \lambda > 0$} .
  \end{cases}
\]
In particular, all off-diagonal elements of\/ $\log(M)$ are
non-negative. When the minimal polynomial of\/ $M$ has degree\/ $2$,
there is no other real logarithm of\/ $M$.
\end{lemma}

\begin{proof}
  The claim on the spectrum is a simple consequence of $\det (M) > 0$
  in conjunction with the spectral radius being an eigenvalue, as
  explained above.
  
  When $\sigma (M)$ is simple and positive, there is precisely one
  real logarithm of $M$ by Fact~\ref{fact:Culver}. It thus suffices to
  verify the above formula. Since $M$ and $\log(M)$ are simultaneously
  diagonalisable in this case, and the claimed formula for the
  logarithm is written as a linear combination of $\one$ and $M$, it
  suffices to check the formula for a diagonal matrix, which means
  that we may assume $M = \diag(\lambda^{}_{+}, \lambda^{}_{-})$
  without loss of generality. But then, we get
  $\log(M) = \diag \bigl( \log(\lambda^{}_{+}),
  \log(\lambda^{}_{-})\bigr)$, as we must.
  
  The case with degenerate spectrum follows from the general one by
  l'Hospital's rule. It gives the correct answer both for
  diagonalisable and non-diagonalisable matrices. Here, a direct
  verification can also be done via the Cayley--Hamilton theorem
  \cite[Thm.~7.3.4]{Lan}. Since
  $M = \lambda \bigl( \one + \lambda^{-1} (M - \lambda \ts \one )
  \bigr)$, the principal matrix logarithm of $M$ can be calculated as
\[
\begin{split}
    \log (M) \, & = \, \log(\lambda) \ts \one + \log
    \bigl( \one + \lambda^{-1} ( M - \lambda \ts \one ) \bigr) 
    \, = \, \log(\lambda) \ts \one \, +
    \sum_{m \geqslant 1} \myfrac{(-1)^{m+1}}{m \ts \lambda^m}
    \ts (M - \lambda \ts \one)^m   \\[1mm]
    & = \, \log(\lambda) \ts \one + \lambda^{-1} 
      (M - \lambda \ts \one)  \, = \, ( \log (\lambda) -1 )
      \one + \lambda^{-1} M \ts ,
\end{split}\]  
where the second line follows because
$(M - \lambda \ts \one)^2 = M^2 - \tr(M) M + \det(M) \ts \one = \nix$ by
Cayley--Hamilton, which implies that this and all higher-order terms
in the series vanish. When $M$ fails to be diagonalisable, it must
have a non-trivial Jordan normal form with a single Jordan block 
$J_2 (\lambda)$, in which case the real logarithm is
again unique by Fact~\ref{fact:Culver}.

Under our assumptions, in all cases, the pre-factor of $M$ is always a
positive number, so $\log(M)$ inherits non-negative off-diagonal
elements from the corresponding property of $M$.
\end{proof}

Let us explore what the result means for sub-stochastic matrices.

\begin{example}\label{ex:2D}
  Consider the matrix
  $M = \left( \begin{smallmatrix} 0.1 & \vt \\ 0 &
      0.2 \end{smallmatrix} \right)$, which is sub-stochastic for
  $0 \leqslant \vt \leqslant 0.9$. For $\vt=0$, it is embeddable via
  $Q = \diag \bigl( \log (0.1) , \log (0.2) \bigr)$. Since $M$ has
  simple, positive spectrum, its real logarithm is unique and given by
  Lemma~\ref{lem:sub-Kendall}, where all off-diagonal elements are
  non-negative, as required for a sub-stochastic generator.
  
  However, the row sums of $\log (M)$ are non-positive only for
  $\vt\leqslant \vt_0 = \frac{\log(10)}{10 \ts \log(2) }\approx
  0.332193$, while the first row sum is positive for $\vt > \vt_0$.
  This follows from the formula in Lemma~\ref{lem:sub-Kendall}.  The
  result means that $\log (M)$ fails to be a sub-stochastic generator
  for $\vt > \vt_0$, and $M$ is not embeddable then. This lines up
  with the case of positive simple spectrum in \cite[Table~1]{BS3} for
  the embeddability of $M_{\mathsf{c}} \in \cM_3$.  \exend
\end{example}

When $M \in \Mat(2,\RR)$ is sub-stochastic with a minimal polynomial
of degree $1$, we have an eigenvalue $\lambda \in (0,1]$ with equal
algebraic and geometric multiplicity, namely $2$. So, its JNF is
$\diag (\lambda, \lambda)$, and it is clear that we must have
$M = \lambda \ts \one_2$. Such an $M$ is always embeddable as
$M = \exp \bigl( \log(\lambda) \ts \one_2 \bigr)$ via its principal
matrix logarithm. The embedding is unique, as no other real matrix
logarithm $R$ of $M$ can be a sub-stochastic generator, by an
application of \cite[Cor.~2.11]{BS3}, because the off-diagonal entries
of $R$ are either both $0$ or have opposite signs.

\begin{example}\label{ex:2D-Jordan}
  Consider a non-negative matrix with non-trivial JNF, namely
  $M = \left( \begin{smallmatrix} 0.5 & \vt \\ 0 &
      0.5 \end{smallmatrix} \right)$ with $\vt > 0$, which is
  sub-stochastic for $\vt \leqslant 0.5$.  Here, $\log(M)$ according
  to Lemma~\ref{lem:sub-Kendall} is sub-stochastic for
  $\vt \leqslant \vt_0 = \frac{1}{2} \log(2) \approx 0.346574$, but
  not for $\vt > \vt_0$. Since the real logarithm of any $M$ with
  $\vartheta > 0$ is unique, we get embeddable as well as
  non-embeddable cases.  \exend
\end{example}

We can wrap up the embedding situation for $d=2$ as follows, where we
only need to look at matrices with positive determinant, because no
other case can be embeddable, by Theorem~\ref{thm:extend}. When
$M\in \cM_{2,\leqslant}$ with $\det(M)>0$ has simple spectrum or fails
to be diagonalisable, the only real matrix logarithm is the one from
Lemma~\ref{lem:sub-Kendall}. The remaining cases were discussed before
Example~\ref{ex:2D-Jordan}, where we also mention that the set of
embeddable matrices is relatively closed within the set of matrices
with positive determinant \cite[Prop.~3]{King}. This gives the
following consequence.

\begin{coro}
  Let\/ $M\in \cM_{2,\leqslant}$ with\/ $\det(M) >0$ have a minimal
  polynomial of degree\/ $2$, so\/ $M$ has simple spectrum or JNF\/
  $J_2 (\lambda)$ with\/ $\lambda \in (0,1)$. Then, $M$ is embeddable
  if and only if all row sums of the principal matrix logarithm from
  Lemma~$\ref{lem:sub-Kendall}$ are non-positive. In this case, the
  embedding is unique.

  The only remaining case is\/ $M = \lambda \ts \one_2$ with\/
  $\lambda \in (0,1]$, which is always embeddable via its principal
  matrix logarithm, $Q= \log(\lambda) \ts \one_2$. This embedding is
  unique as well.  \qed
\end{coro}

In particular, any $M \in \cM_2 \subset \cM_{2,\leqslant}$ with
$\det (M) >0$ is uniquely embeddable, which covers all cases for $d=2$
according to Kendall's theorem mentioned earlier. Note that no
embeddable case with negative eigenvalues can occur for $d=2$.

\subsection{Results for \texorpdfstring{$d=3$}{\textit{d}=3}}

Let us now discuss the embedding of matrices $M \in \cM_{3,\leqslant}$
on the basis of the known embedding results for Markov matrices with
$d=4$, as completely summarised in \cite[Table~2]{BS3}. The increased
difficulty emerges from the possible appearance of a complex conjugate
pair of eigenvalues or a pair of negative eigenvalues.

When $M\in\cM_{3,\leqslant}$ is itself Markov, we are in the situation
of the standard embedding problem, which is classified in
\cite[Table~1]{BS3}. Let us thus focus on
$M \in \cM_{3,\leqslant} {\setminus} \ts \cM_3$, with $\det (M)
>0$. The cases are best distinguished via the minimal polynomials of
$M$ and $M_{\mathsf{c}}$.

For all sub-stochastic matrices, the principal matrix logarithm exists
via a convergent series, see \cite[Eq.~(1.1) and Thm.~1.31]{Higham},
and is real. So, let us first look at those matrices
$M \in \cM_{3,\leqslant}$ with positive spectrum where no other real
matrix logarithm exists, which can be identified via
Fact~\ref{fact:Culver}. Here, all eigenvalues lie in the half-open
interval $(0,1]$, where we have to observe the restriction from
Corollary~\ref{coro:JNF}. Let $q$ be the minimal polynomial of
$M$. When $\deg (q) = 3$, the matrix is cyclic, which simplifies
matters. Here, we can have simple spectrum, JNF
$\lambda' \oplus J_2 (\lambda)$ with $\lambda \ne \lambda'$ or JNF
$J_3 (\lambda)$, both with $\lambda < 1$. Then, we get embeddability
if and only if the principal matrix logarithm is a sub-stochastic
generator.  The latter is directly accessible by observing that, with
$A = M - \one$, one has
\[
  \log (M) \, = \, \log(\one + A) \, = \,
  \sum_{m=1}^{\infty} \myfrac{(-1)^{m-1}}{m} A^m
  \, = \, a^{}_{0} \one + a^{}_{1} A + a^{}_{2} A^2 ,
\]
where the last step follows from the Cayley--Hamilton theorem. Note
that $a^{}_{0}$ need not be $0$ because $A$ need not have $0$ row
sums. The coefficients of this polynomial can now be computed via the
spectral mapping theorem for the eigenvalues \cite[Thm.~9.4.6]{Lan},
which gives a system of three linear equations with the Vandermonde
matrix (for simple spectrum) or its confluent counterpart (for Jordan
blocks).  This system always has a unique solution; compare
\cite{BS2}.

The remaining case of a sub-stochastic matrix $M$ with unique real
logarithm has JNF $\lambda \oplus J_2 (\lambda)$ with $\lambda < 1$,
since $\lambda = 1$ is impossible by Corollary~\ref{coro:JNF}. Note
that $M$ is not cyclic, so has a non-Abelian commutant; see
\cite[Thm.~12.4.1]{Lan}. However, the principal logarithm of $M$ has
the same commutant as $M$ itself, and thus causes no issues, and we
have embeddability if and only if the principal matrix logarithm is a
sub-stochastic generator. A concrete example is
$M = 0.5 \oplus \left( \begin{smallmatrix} 0.5 & \vartheta \\ 0 &
    0.5 \end{smallmatrix} \right)$ in obvious extension of
Example~\ref{ex:2D-Jordan}, with the same parameter regions and
embedding conditions as stated there. Let us sum up as follows.

\begin{coro}
  Let\/ $M \in \cM_{3,\leqslant}$ have positive spectrum. If\/ $M$ is
  cyclic or if it has JNF\/ $\lambda \oplus J_2 (\lambda)$ for some\/
  $\lambda \in (0,1)$, the matrix\/ $M$ is embeddable if and only if
  its principal matrix logarithm is a sub-stochastic generator. In
  this case, the embedding is unique.  \qed
\end{coro}

Let us next look at the truly sub-stochastic cases that are always
embeddable, but possibly without uniqueness. Here, one JNF is
$\lambda \ts \one_3$ with $\lambda \in (0,1)$, which actually means
$M = \lambda \ts \one_3$. The principal matrix logarithm of $M$ is
$Q = \log (\lambda) \one_3$, which is a sub-stochastic generator.
Here, no further real logarithm of the required generator form exists.
Indeed, any other real logarithm $R$ of $M$, due to $M = \ee^R$, must
have spectrum
$\sigma (R) = \{ \epsilon, \epsilon + 2 \pi k \ii, \epsilon - 2 \pi k
\ii \}$ for some $0 \ne k\in \ZZ$, with $\epsilon = \log
(\lambda)<0$. The JNF of $R$ then is
$\epsilon \ts \one_3 + 2 \pi k \ts \diag (0, \ii, -\ii)$. This JNF can
only belong to a sub-stochastic generator when there exists some
$U \in \mathrm{GL} (2,\CC)$ so that $U \diag (\ii, -\ii) U^{-1}$ is
real and has non-negative off-diagonal elements. However, in line with
\mbox{\cite[Cor.~2.11]{BS3}}, there is no solution for $k\ne 0$
because the off-diagonal terms then have opposite signs. As an
extension of \cite[Fact~2.9]{BS3}, we thus have the following result.

\begin{fact}\label{fact:simple-diag}
  The sub-stochastic matrix\/ $M = \lambda \one_3$ with\/
  $\lambda \in (0,1]$ is always embeddable via its principal matrix
  logarithm, which is\/ $Q = \log (\lambda) \one_3$, and the embedding
  is unique.  \qed
\end{fact}

Further cases are the sub-stochastic matrices $M$ with JNF
$1 \oplus J_2 (\lambda)$ or $\diag(1,1,\lambda)$, both with
$\lambda \in (0,1)$.  In both cases, embeddability is given via the
principal matrix logarithm.  Since such an $M$ commutes with matrices
outside the polynomial ring $\RR [M]$, further embeddings can exist.
These cases are covered by Fact~2.9 and Lemmas~4.1 and 4.3 from
\cite{BS3}, applied to the extension $M_{\mathsf{c}}$. The result can
be formulated as follows.

\begin{coro}\label{coro:3-always}
  Let\/ $M \in \cM_{3,\leqslant}$ have JNF\/ $1 \oplus J_2 (\lambda)$
  or\/ $\diag (1,1,\lambda)$, both with\/ $\lambda \in (0, 1) $. Then,
  $M$ is always embeddable via its principal matrix logarithm, which
  is a sub-stochastic generator, but this embedding need not be
  unique.  \qed
\end{coro}

The remaining cases (with positive determinant) are slightly more
complicated. We cannot have any negative eigenvalue with odd algebraic
multiplicity in an embeddable matrix, and we cannot have
$J_2 (\lambda)$ with $\lambda < 0$, both as a result of
Fact~\ref{fact:Culver}. But we can have JNF
$\diag (\rho, \lambda, \lambda)$ with $\rho \in (0,1]$ and
$\lambda < 0$, where then $|\lambda| < \rho$, and we can have
$\diag (\rho, \vartheta, \overline{\vartheta} \ts )$ with
$\rho \in (0,1]$ and $\vartheta \in \CC \setminus \RR$, then with
$\lvert \vartheta \rvert \leqslant \rho$. Let us first note the
following.

\begin{lemma}\label{lem:irred}
  Let\/ $\rho \in (0,1]$ be fixed and let\/ $M \in \cM_{3,\leqslant}$
  have JNF\/ $\rho \oplus \lambda \one_2$ with\/ $\lambda < 0$ or
  JNF\/ $\rho \oplus \diag (\vartheta, \overline{\vartheta}\ts )$
  with\/ $\vartheta \in \CC\setminus \RR$. Then, $M$ is irreducible.
\end{lemma}

\begin{proof}
  Let us assume to the contrary that $M$ is reducible. Then, by a
  suitable permutation of the states, we can bring $M$ to its normal
  form for non-negative matrices, which is of upper block-triangular
  form with all blocks being non-negative matrices, and the diagonal
  blocks being irreducible (or $0$, which is impossible here); see
  \cite[Sec.~15.5]{Lan}.

  This normal form cannot be upper triangular, because this would
  imply all eigenvalues to be positive, in contradiction to our
  assumptions. So, it has to have one singleton diagonal block, which
  must then contain a positive eigenvalue, for which $\rho$ is the
  only choice. The other diagonal block, $N$ say, has to be
  two-dimensional and irreducible. Clearly, $N$ is non-singular and
  has its spectral radius as a positive eigenvalue.

  On the other hand, due to the block structure in conjunction with
  the argument used for Eq.~\eqref{eq:char-poly}, the remaining
  eigenvalues of $M$, which are either negative or complex, must be
  the eigenvalues of $N$. So, we get a contradiction, and $M$
  reducible is impossible.
\end{proof}

By Fact~\ref{fact:radius}, a matrix $M \in \cM_{3,\leqslant}$ of JNF
$\diag(1, \lambda, \lambda)$ with $\lambda < 0$ or of JNF
$\diag(1, \vartheta, \overline{\vartheta} \ts )$ with
$\vartheta \in \CC\setminus \RR$ must be Markov, so the embeddability
of $M$ for these two cases is fully covered by \cite[Props.~3.8 and
3.11]{BS3}, as mentioned earlier; see also \cite[Table~1]{BS3}.  It
thus remains, for $\rho \in (0,1)$, to consider the JNFs
$\diag(\rho,\lambda,\lambda)$ with $\lambda<0$ and
$\diag(\rho,\vartheta,\overline{\vartheta}\ts)$ with
$\vartheta \in \CC\setminus \RR$. Both occur as a consequence of
\cite[Thm.~2]{LL}, and can be treated via \cite[Lemma~4.5 and
Prop.~4.9]{BS3}. Here, we only present one characteristic example
each, with some focus on multiple embeddings.

\begin{example}\label{ex:sub-real}
  Consider the matrix
\[
     M \, = \, \myfrac{1}{3} \begin{pmatrix}  r - 2 s & r + s & r + s \\
        r + s & r - 2 s & r + s \\ r + s & r + s & r - 2 s \end{pmatrix}
\]  
with $0 < r \leqslant 1$ and $0 < s \leqslant \frac{r}{2}$. It is
sub-stochastic, with spectrum $\sigma (M) = \{ r, -s, -s\}$, and is
Markov for $r=1$. In the latter case, $M$ is known to be doubly
embeddable for $s = \ee^{- \pi \sqrt{3}}$ as
$M = \exp \bigl( \frac{2 \pi}{\sqrt{3}} Q \bigr) = \exp \bigl( \frac{2
  \pi}{\sqrt{3}} Q^{\trans} \bigr) $ with the cyclic Markov generator
\[
    Q \, = \, \begin{pmatrix} -1 & 1 & 0 \\ 0 & -1 & 1 \\
              1 & 0 & -1 \end{pmatrix} ,
\]  
which commutes with its transpose; see \cite{BBS} and references
therein for more, as well as for the history of this important
example.  Fixing $r=1$ and decreasing $s$, one finds more and more
additional embeddings, as derived in \cite[Rem.~3.9]{BS3} and
\cite[App.~A]{BBS}.
   
Now, consider $Q^{}_{\alpha}\defeq Q - \alpha \one_3$ for
$\alpha > 0$, which is a sub-stochastic generator that satisfies
 \[
      \exp \Bigl( \myfrac{2 \pi}{\sqrt{3}} \ts Q^{}_{\alpha} \Bigr) 
      \, = \, \exp \Bigl( \myfrac{2 \pi}{\sqrt{3}} \ts Q^{\trans}_{\alpha} 
      \Bigr) \, = \,  \ee^{-\tau} M
 \]
 with $\tau = \frac{2 \pi \ts \alpha}{\sqrt{3}} > 0$, so
 $\ee^{-\tau} < 1$. We thus obain an embeddable sub-stochastic matrix
 of the above kind, with $r = \ee^{-\tau}$ and
 $s = \ee^{-\tau - \pi \sqrt{3}}$. When $\alpha = 0$ and the parameter
 $s$ is decreased, embeddability of $M$ remains valid. What is more,
 at a series of values, additional generators emerge, and we get
 multiple embeddings.  Their number increases without bound as
 $s \text{\raisebox{1.5pt}{$\,{\scriptscriptstyle \searrow}\,$}} 0$,
 and this, via choosing $\alpha > 0$ again, extends to the
 substochastic matrices as well; we refer to \cite{BS3, BBS} and leave
 further details to the interested reader. \exend
\end{example}

\begin{example}\label{ex:sub-complex}
  Let $M \in \cM_{3,\leqslant}$ have spectrum
  $\sigma (M) = (\rho, \vartheta, \overline{\vartheta} \ts )$ with
  $\rho \in (0,1)$ and $\vartheta \in \CC\setminus \RR$, where
  $| \vartheta | \leqslant \rho$. Let us look at a concrete case,
  namely $\vartheta = \tau \ii$ with $\tau > 0$. By the methods from
  \cite{LL}, one possibility for $M$ then is
\[
  M \, = \, \myfrac{1}{3} \begin{pmatrix}
    \rho & \rho + \tau \sqrt{3} & \rho - \tau \sqrt{3} \\
    \rho - \tau \sqrt{3} & \rho &  \rho + \tau \sqrt{3} \\ 
    \rho + \tau \sqrt{3} & \rho - \tau \sqrt{3} 
    & \rho \end{pmatrix} ,
\]
which is sub-stochastic for $\tau \leqslant \rho/\sqrt{3}$,
and has determinant $\det(M)=\tau^2 \rho > 0$.
  
The corresponding extension matrix $M_{\mathsf{c}} \in \cM_4$, which
has simple spectrum, is embeddable when
$Q^{}_{\mathsf{c}} = \alpha A^{}_{\mathsf{c}} + \beta A_{\mathsf{c}}^2
+ \gamma A_{\mathsf{c}}^3$ is a generator, where
$A_{\mathsf{c}} = M_{\mathsf{c}} - \one_4$ and $\alpha, \beta, \gamma$
are determined via the spectral mapping theorem; compare
\cite[Sec.~4.3]{BS3}. By construction, $Q^{}_{\mathsf{c}}$ has zero
row sums, so we only need to check non-negativity of its off-diagonal
entries. Using some computer algebra tools, one finds a sufficient
condition in the form of a double inequality, namely
\begin{equation}
  \log \Bigl( \myfrac{\tau}{\rho}  \Bigr)
  \, \leqslant \, \myfrac{(4 k + 1) \ts \pi \sqrt{3}}{2}
  \, \leqslant \,  \log \Bigl( \myfrac{\rho}{\tau} \Bigr) ,
\end{equation}
where $k\in\ZZ$ is the free parameter from \cite[Prop.~4.9]{BS3}.

For $k=0$, both inequalities are satisfied when
$\tau \leqslant \rho \ts \ee^{- \pi \sqrt{3}/2}$, while
$0 \ne k\in\ZZ$ requires
$\tau \leqslant \rho \ts \ee^{- \mathrm{sgn} (k) \ts (4k+1) \pi
  \sqrt{3}/2}$. One can also extract the corresponding sub-stochastic
generator $Q^{(k)}$ via the upper-left block of
$Q^{(k)}_{\mathsf{c}}\!$, for any $k\in\ZZ$, which gives
\[
  Q^{(k)} \, = \, \myfrac{1}{3} \begin{pmatrix}
     \log ( \tau^2 \rho) & \log \bigl(\frac{\rho}{\tau} \bigr)
     + \frac{(4k+1)}{2} \pi \sqrt{3}   & 
       \log \bigl( \frac{\rho}{\tau}\bigr)
         - \frac{(4k+1)}{2} \pi \sqrt{3} \\
   \log \bigl(\frac{\rho}{\tau} \bigr) - \frac{(4k+1)}{2} \pi \sqrt{3}
   & \log(\tau^2 \rho) & \log \bigl( \frac{\rho}{\tau}\bigr)
        + \frac{(4k+1)}{2} \pi \sqrt{3} \\
   \log \bigl(\frac{\rho}{\tau} \bigr) + \frac{(4k+1)}{2} \pi \sqrt{3}
   & \log \bigl( \frac{\rho}{\tau}\bigr) - \frac{(4k+1)}{2} \pi \sqrt{3}
   & \log(\tau^2 \rho)   \end{pmatrix}.
\]
For fixed $\rho$, we thus have a family of Markov matrices
$M = M(\tau)$ with an ever-increasing number of embeddings as $\tau$,
and hence $\det(M)$, decreases and tends to $0$. This also works
for the limiting case of $\rho = 1$, where $M$ is Markov.  \exend
\end{example}

\section*{Acknowledgements}

It is our pleasure to thank Ellen Baake and Jeremy Sumner for fruitful
discussions. KS is grateful to the Research Centre for Mathematical
Modelling (RCM$^2$) of Bielefeld University for support during a
research visit in summer 2026, where part of this work was done.

\end{document}